\documentclass[11pt]{article}
\title{On gradient flows initialized near maxima}
\author{Mohamed-Ali Belabbas}

\usepackage{amssymb}
\usepackage{latexsym}
\usepackage{amsmath}
\usepackage{amsthm}
\usepackage{amscd}
\usepackage{mathtools}
\mathtoolsset{showonlyrefs}

\usepackage{tikz}
\usepackage{lmodern}

\usepackage{cite} 
\usepackage{setspace}
\usepackage[latin1]{inputenc}
\usepackage{amsmath}
\usepackage{amsthm}
\usepackage{color}
\usepackage{amsfonts}
\usepackage{amssymb}
\usepackage{graphicx}
\usepackage{epsfig}
\usepackage{url}
\usepackage{tikz-cd} 
\usepackage{bbm} 
\usepackage{microtype}

\newtheorem{theorem}{Theorem}

\newtheorem{proposition}{Proposition}
\newtheorem{corollary}[theorem]{Corollary}
\newtheorem{lemma}{Lemma}

\newtheorem{remark}{Remark}

\newcommand{\diag}{\operatorname{diag}}

\newcommand{\R}{\mathbb{R}}

\newcommand{\0}{\mathbf{0}}

\newcommand{\Span}{\operatorname{span}}

\newcommand{\grad}{\nabla}
\newcommand{\crit}{\operatorname{Crit}}

\newcommand{\cM}{\mathcal{M}}
\newcommand{\cN}{\mathcal{N}}
\newcommand{\cG}{\mathcal{G}}
\newcommand{\cU}{\mathcal{U}}
\newcommand{\cL}{\mathcal{L}}
\newcommand{\cF}{\mathcal{F}}

\newcommand{\hyp}{\operatorname{Hyp}}

\newcommand{\Diff}{\mathrm{Diff}}

\newcommand{\ind}{\operatorname{ind}}
\newcommand{\epi}{\operatorname{Epi}}

\newcommand{\codim}{\operatorname{codim}}

\newtheorem{definition}{Definition}
\newcounter{para}

\date{}

\begin{document}

\maketitle

\begin{abstract}
 Let $(M,g)$ be a closed Riemannian manifold, and let $F:M \to \R$ be a smooth function on $M$.  We show the following holds generically for the function $F$: for each  maximum $p$ of $F$, there exist two minima, denoted by $m_+(p)$ and $m_-(p)$, so that the gradient flow initialized at a random point close to $p$ converges to either $m_-(p)$ or $m_+(p)$ with high probability. The statement also holds for $F \in C^\infty(M)$ fixed and a generic metric $g$ on $M$. We conclude by associating to a given a generic pair $(F, g)$ what we call its  max-min graph, which captures the relation between minima and maxima derived in the main result.
\end{abstract}

\section{Introduction}

A major challenge in non-convex optimization is to understand to which minimum the gradient flow of a differentiable function converges. 
Indeed, this minimum depends on the initialization of the gradient flow, and understanding how this initialization impacts the gradient trajectory requires a global analysis that is in general difficult. To sidestep these difficulties, stochastic methods such as simulated annealing~\cite{pincus1970letter} have been put forward, with the goal of using stochasticity to decouple the initialization of the flow from its convergence point~\cite{brockett1997oscillatory, raginsky2017non}. However, this comes at the cost of slower convergence times and reliance on heuristics to set the value of some parameters.  Moreover, there are scenarios, e.g. arising in learning theory~\cite{MorseIR20,implicitregmatr2017}, in which a deterministic initialization   is required.  In this paper, we study the qualitative behavior of gradient flows. More precisely, we show that regardless of the number of minima of $F$, for each maximum $p$ of $F$, there exists {\em two} minima, not necessarily distinct, so that the gradient flow initialized near $p$ converges to these minima with very high probability. Based on this characterization, we can naturally assign a graph to each generic pair $(F,g)$; we refer to it as a max-min graph and discuss some of its basic properties in the last section, leaving its complete analysis to a forthcoming publication.

\subsection{Statement of the main result} Let $(M,g)$ be a smooth closed Riemannian manifold and $F:M \to \R$ be a smooth function. We denote by $\grad^g F$ the {\em gradient vector field} of $M$ for the inner product $g$, which is defined by the equation\begin{equation}\label{eq:defgrad}
 g(\grad^g F,X)=dF \cdot X	 \mbox{ for all } x \in M, X \in T_xM,
 \end{equation}
see~\cite{helmke2012optimization,bloch1992completely} for examples. We omit the exponent $g$ when the metric is clear from the context. Given a differentiable vector field $f(x)$ on $M$, we denote by $e^{tf}x$ the one-parameter group of diffeomorphisms with infinitesimal generator $f$. Namely, we set \begin{equation}\label{eq:basiceq}e^{\cdot f}\cdot:\R \times M \to M: (t,x) \mapsto e^{t f} x\end{equation} to be the solution at time $t$ of the Cauchy problem \begin{equation}\label{eq:odediffeo}\dot y = f(y), \quad y(0) = x.\end{equation} The {\bf gradient flow} of $F$ at time $t$ for the metric $g$ is the map $ x \mapsto e^{-t \nabla^g F}x$. 
 We also write $e^{[-t,t] f}x$ to denote the solution of~\eqref{eq:odediffeo} between time $-t$ and $t$. For a subset $B \subset M$, we let $e^{[-t,t] f} B:= \bigcup_{x \in B} e^{[-t,t]f}x.$
 
 We denote by $\cM$ the space of smooth Riemannian metrics on $M$ and by  $C^\infty(M)$ the space of smooth real-valued functions on $M$. We endow these spaces with the Whitney $C^k$-topology, for any $k \geq 3$ fixed~\cite{hirsch2012differential}. 
Given a topological space $X$, we say that a subset $Y \subseteq X$ is {\bf residual} if it is a countable intersection of open dense subsets $Y_i$ of $X$, i.e.,  $Y = \bigcap _{i=1}^\infty Y_i$. 
A subset $A \subseteq X$ is called {\bf generic} if it contains a residual set.  Finally, we say that $X$ is a {\bf Baire} space if generic subsets of $X$ are also dense in $X$. The sets $\cM$ and $C^\infty(M)$, equipped with the Whitney $C^k$-topology, are  Baire spaces.

We are now in a position to state the main result of this paper. Let $d:M \times M \to \R$ be a Riemannian distance function and denote by $B_\delta(p)$ the ball of radius $\delta$ centered at $p$ for the distance $d$: 
\begin{equation}\label{eq:maineq}B_\delta(p):=\{x \in M \mid d(x,p)\leq\delta\}.\end{equation}Note that $d$ is not necessarily the distance induced by the metric $g$.
 For $m_1,m_2 \in M$, let  $W_\delta(p,m_1,m_2)$ be the set of points in $B_\delta(p)$ belonging to trajectories converging to either $m_1$ or $m_2$:
 $$W_\delta(p,m_1,m_2) :=\left\lbrace x \in B_\delta(p) \mid \lim_{t \to \infty}e^{-t \grad F}x \in \{m_1,m_2\}\right\rbrace.$$ In terms of the stable manifolds $W^s(m_i)$ (see~\cite{banyaga2013lectures} or below for a definition), we have $W_\delta(p,m_1,m_2) = B_\delta(p) \cap (W^s(m_1) \cup W^s(m_2)).$ The main theorem is:
 
\begin{theorem}\label{th:main} 
 Let $F \in C^\infty(M)$ a Morse function on a smooth closed Riemannian manifold $(M,g)$. Let $\mu$ be a measure on $M$ induced by a smooth positive density and let $d:M \times M \to [0,\infty)$ be any  Riemannian distance function. Then generically for $g$ (resp. generically for $F$), the following holds: For any maximum $p$, there exists two minima $m_+(p),m_-(p)$ with the property that for all $\varepsilon>0$, there is $\delta>0$  such that  \begin{equation}\label{eq:theorem}\mu\left(W_\delta(p,m_+,m_-)\right)	 \geq (1-\varepsilon) \mu(B_\delta(p)).\end{equation}
\end{theorem}

We make a  few comments on the Theorem. The minima $m_-(p)$ and $m_+(p)$ are not necessarily distinct; the gradient flow of the height function on a sphere provides a simple example of this fact. The proofs below  hold for $F$  of class  $C^3$ and $g$ of class $C^2$. The minimal differentiability requirement stem from the use of a linearization theorem of Hartman, see Th.~\ref{th:hartman} below. In fact, since we use this theorem locally, one could even relax the hypotheses to include functions and metrics that are of class $C^3$ and $C^2$ around local maxima only. The results also hold for Morse functions $F:K\subset \R^n \to \R$, where $K$ is any compact set so that $\grad F$ evaluated on $\partial K$ points outside of $K$ (said more precisely,  $e^{-t\grad F}K \subset K$ for $t\geq 0$.)
\subsection{Overview of the proof}  The first step of the proof is to exhibit a necessary condition on the gradient of $F$ ensuring that~\eqref{eq:theorem} holds for a maximum and some pair of minima of $F$. To this end, we introduce the notion of {\it principal flow lines} of a maximum of $F$. After having defined the principal flow lines, we show in Proposition~\ref{prop:main1} that if they meet a condition described below, then~\eqref{eq:theorem} holds---we will say that a maximum of $F$ is simple if its principal flow lines meet this condition. Finally, we will show in Proposition~\ref{prop:main2} that gradient flows with simple maxima are generic. We will prove genericity in terms of the choice of $g$ for a fixed Morse function $F$, and reciprocally genericity for a smooth $F$ given a metric $g$.

\subsection{Terminology and conventions}We denote  by $e_1,\ldots, e_n$ the canonical basis of $\R^n$. We let $S^{n-1}_r(p) \subset \R^n$ be the unit sphere of dimension $n-1$, radius $r$ and centered at $p$.  We let $D^n_r(p)\subset \R^n$ be the closed ball of radius $r$ centered at $p$ and $D_r^{n,+}(p)$ be the upper ``half-ball" $$D_r^{n,+}(p):=\{x \in \R^n\mid \|x-p\|\leq r \mbox{ and } e_1^\top (x-p) \geq 0 \}.$$ 
For $x=(x_1,\ldots,x_n)$, we define the projections $$\pi_1:\R^n \to \R: x \mapsto x_1\mbox{ and }\pi_{-1}:\R^{n} \to \R^{n-1}:x \mapsto (x_2,\ldots,x_n).$$ For a Morse function $F$ with a critical point $p$, we denote by $\ind(p)$ the Morse index of $F$ at $p$. Given a map $\varphi:M \to N$, we denote by $\varphi_*$ its pushfoward~\cite{lee2013smooth}.

Recall that two submanifolds $M_1, M_2 \subset M$ intersect transversally at $x \in M_1 \cap M_2$ in $M$ if $T_xM_1 \oplus T_xM_2 = T_x M$. For a vector field $f$ on $M$, we say that $M_1$ and $f$ are transversal at $x \in M_1$ if $T_x M_1 \oplus \Span\{f(x)\}=T_xM$. We shall use transversality and appeal to the jet transversality theorem at various places in the proof. We refer to~\cite{hirsch2012differential} for an introduction.
We will use throughout the paper the letter $c$ to denote a real constant, with the understanding that the value of $c$ can change during a derivation.


\section{Preliminaries}
 We let $F \in C^\infty(M)$; a {\bf critical point} of $F$ is a point $x$ so that $dF(x) = 0$. Their set is denoted by $\crit F$. We say that a critical point is {\it non-degenerate} if the symmetric matrix $\frac{\partial^2 F}{\partial z^2}(p)$, where $z$ are coordinates around $p$, is invertible.
 A function with non-degenerate critical points is called a {\bf  Morse function}~\cite{milnor2016morse}.  We call the {\bf  Morse index} or index of a critical point $p$ the number of {\it negative} eigenvalues of $\frac{\partial^2 F}{\partial z^2}(p)$.
  If $F$ is Morse, it is easy to show that its critical points are isolated (see, e.g.,~\cite[Lemma 3.2]{banyaga2013lectures}) and thus, since $M$ is compact, they are finite in number. We denote by $\crit_i F$ the set of critical points of $F$ of index $i$. Consequently, the set $\crit_n F$ is the set of maxima of $F$, and $\crit_0 F$ the set of minima.

Given a metric $g \in \cM$ (resp. $F \in C^\infty(M)$) and a property $S$ (e.g. $F$ being Morse), we say that there exist $h \in \cM$ with property $S$ {\bf arbitrarily close} to $g$ if every Whitney  open set containing $g$ also contains an element $h$ with property $S$. For example, if $F$ is a smooth function, it is well-known that there exist Morse functions arbitrarily close to $F$.

The {\bf stable manifold} $W^s(p,g)$  of a critical point $p$ is defined as 
\begin{equation}
\label{eq:defstabman}
W^s(p,g):= \{ x \in M \mid \lim_{t \to \infty}e^{-t \nabla^g F}(x) = p\}; 
\end{equation} 
when the metric is obvious from the context, we omit it and simply write $W^s(p)$.	 Similarly, we define the {\bf unstable manifold} of $p$ as
$$W^u(p,g):= \{ x \in M \mid \lim_{t \to -\infty}e^{-t \nabla^g F}(x) = p\}.
$$
The stable manifold theorem (for Morse functions) states (e.g.,~\cite[Theorem 4.2]{banyaga2013lectures}) that $W^s(p)$ is a smoothly {\it embedded} open-ball of dimension $n-\ind(p)$ in $M$. We furthermore have the following decomposition of $M$ afforded by stable (resp. unstable) manifolds of the critical points of a Morse function $F$:
$$M = \coprod_{p \in \crit F} W^u(p) = \coprod_{p \in \crit F} W^s(p).$$

We will use a result of Hartman~\cite{hartman1960mex,newhouse2017}  which generalizes the  Poincar\'e-Dulac theorem on the linearization of analytic vector fields near a singularity~\cite{arnold77geometrical}. It  provides conditions under which  a diffeomorphism is locally $C^1$-conjugate to its linearization at a fixed point:

\begin{theorem}[Hartman]\label{th:hartman}
Let $U$ be an open subset of $\R^n$, $0 \in U$ and  $f:U \to \R^n$ be a $C^2$ vector field with $f(0)=0$. Assume that all the eigenvalues of $A:=\frac{\partial f}{\partial x}(0)$ have a negative real part. Then there exists  open neighborhoods $V \subset U$, and $W$ of the origin,  and a $C^1$ diffeomorphism $\psi:V \to W$ so that for $z=\psi(x)$, the differential equation $\dot x = f(x)$ is conjugate to $\dot z= Az$.
\end{theorem}

We will rely on the  following two simple results, whose proofs are omitted, to apply Theorem~\ref{th:hartman} to gradient vector fields.  

\begin{lemma}\label{lem:techH}
Let $F$ be a smooth Morse function and $p \in \crit F$. Let $(\varphi,U)$ be a chart so that $\varphi(p)=0$. Denote by $H^g_\varphi(x)=d(\varphi_* \grad^g F)$ the Jacobian matrix of $\grad^g F$ expressed in the coordinate chart $(\varphi,U)$. Then $H^g_\varphi(0)$ is diagonalizable and has real eigenvalues, which are independent from $\varphi$.  Furthermore, the number of negative eigenvalues of $H^g_\varphi(0)$ is equal to the Morse index of $p$. 
\end{lemma}
Since the eigenvalues of $H^g_\varphi(0)$ are independent of the chart $\varphi$, we will simply refer to the eigenvalues of $H^g(0)$.
The following Corollary provides a normal form for gradient flows around maxima (or minima):
\begin{corollary}\label{cor:diagsys}
	Let $F$ be a smooth Morse function on the Riemannian manifold $(M,g)$ and let $\grad F$ be its gradient. For any $p \in \crit_n F$, there exists a chart $(\varphi,U)$ with $\varphi(p)=0$ so that the gradient flow equation $\dot x = -\grad F$ is $C^1$-conjugate to $\dot z = - \Lambda z$  in the coordinates $z=\varphi(x)$, where $\Lambda=\diag(\lambda_1,\ldots,\lambda_n)$, with $\lambda_1 \leq \lambda_2 \leq \cdots \leq \lambda_n <0$.
\end{corollary}

\section{Proof of the main result}
We start by describing the intersection of stable manifolds of $\nabla F$ with submanifolds of $M$. The result is needed for the proofs of Propositions~\ref{prop:main1} and~\ref{prop:main2}. The topology on subspaces of $M$ is the usual subspace topology.

\begin{lemma}\label{lem:bounstabclosed}
Let $(M,g)$ be a closed Riemannian manifold and $F$ a smooth function. Let $S$ be an embedded submanifold of codimension one in $M$ that is everywhere transversal to $\grad^g F$ and set $M_0 := \bigsqcup_{q \in \crit_0 F} W^s(q)$. Then $M_0^S:=M_0 \cap S$ is open dense in $S$.
\end{lemma} 
\begin{proof}
	Recall the stable manifold decomposition of $M$: $$M= \bigsqcup_{q \in \crit F} W^s(q)$$ where each stable manifold $W^s(q)$ is a smoothly embedded open ball of dimension $n-\ind(q)$.  When $\ind(q)=0$, the embedding is  also a submersion and thus an open map. Hence, for $q \in \crit _0 F$, $W^s(q)$ is open in $M$  and $M_0$ is also open, since it is a union of open sets.  Set $$M_1:=M-M_0=\bigsqcup_{q \in \crit F\mid \ind(q)\geq 1} W^s(q).$$ Then $M = M_0 \sqcup M_1$ and $M_1$ is closed. 
		Set $M_1^S:=M_1 \cap S$, then $M_1^S$ is closed in $S$ and  we have $S = M_0^S \sqcup M_1^S$.  Hence  $M_0^S$ is open in $S$ as claimed.  

It remains to show that $M_0^S$ is dense in $S$ or, equivalently, that  $M_1^S$ has an empty interior in $S$. To see this,  first recall that $M_1$ is the disjoint union of embedded open balls of dimension at most $n-1$, and thus by Sard 's theorem, $M_1$'s  interior in $M$ is empty. Now assume by contradiction that there exists a non-empty open set $U \subset M_1^S$, and let $x_0 \in U$. 
Let $B \subset U$  be an embedded  closed ball of dimension $n-1$ properly containing $x_0$. Because $\grad^g F$ is transversal to $S$, 
for $\varepsilon>0$ small enough, $$B_1:=e^{[-\varepsilon,\varepsilon] \grad^gF} \cdot B$$ is diffeomorphic to $[-\varepsilon,\varepsilon] \times B$. Thus  there exists an open neighborhood of $x_0$ in $M$ contained in $B_1$. But since $B \subset M_1^S \subset M_1$ and $M_1$ is invariant under the gradient flow, then $B_1 \subset M_1$ and $M_1$ has a non-empty interior in $M$, which is a contradiction. In conclusion, $M_1^S$ is a closed set with empty interior in $S$. Its complement $M_0^S$ is then open dense in $S$ as claimed. 
	 \end{proof} 

\begin{remark}
	Lemma~\ref{lem:bounstabclosed} can be simplified under the additional assumption that $\grad^gF$ is a {\it Morse-Smale} vector field, i.e., under the additional assumption that the stable and unstable manifolds of $\grad^gF$ intersect transversally. With this additional assumption, one can obtain as a consequence of the $\lambda$-Lemma~\cite[Lemma 2.7.1]{palis1982geometric} that the closure of $M_0$ is {\it equal} to $M_1$ (see also~\cite[Chapter 2]{shub2013global}).
\end{remark}

\subsection{Principal flow lines and simple gradients}
A smooth curve $\gamma_t:\R \to M$ is a trajectory of the gradient flow of $F$ (resp. gradient ascent flow of $F$) if it satisfies $\dot \gamma(t) = -\grad F(\gamma(t))$ (resp. $\dot \gamma_t = \grad F(\gamma_t)$) for all $t \in \R$.  Since $F$ is Morse, it is well known that $\lim_{t \to \pm \infty}  \gamma_t \in \crit F$. We introduce the following definition:

\begin{definition}\label{def:tangentially}Let $(M,g)$ be a smooth Riemannian manifold and $\gamma_t$ a smooth curve in $M$. We say that $\gamma_t$ {\bf reaches $p \in M$ tangentially to $v \in T_pM$} if
\begin{enumerate}
\item $\lim_{t \to \infty} \gamma_t = p.$
\item $\lim_{t \to \infty} \frac{\dot \gamma_t}{\|\dot \gamma_t\|}$ exists and is equal to $v$ 
\end{enumerate}
\end{definition}
The existence of the limit in condition 2 of Def.~\ref{def:tangentially}, under the assumption that $\nabla F$ be analytic, is the content of {\em Thom's generalized gradient conjecture}~\cite{kurdyka2000proof}. While we can construct smooth gradients for which this limit does not exist, we show below in Lemma~\ref{lem:existenceprinc} that when $F$ is Morse, its existence can easily be shown along what we call the principal flow lines.

We now define a class of gradient vector fields for which the main inequality~\eqref{eq:maineq} holds. We call them gradients with simple maxima. In order to define them, we first introduce the notion of {\it principal flow line} of a maximum of $\grad F$.

\begin{definition}[Principal flow lines]\label{def:princ}
Let $F$ be a smooth Morse function with gradient vector field $\grad^g F$ and  $p \in \crit_n F$. Denote by $H^g(p)$  the linearization of $\grad F$ at $p$ and let $v \in T_pM$ be a vector in the eigenspace of the smallest eigenvalue of $H^g(p)$. We say that a trajectory is a {\bf principal flow line of $\grad F$ at $p$} if it is a trajectory of the gradient ascent flow that reaches $p$ tangentially to $v$.
\end{definition}
We have the following result:

\begin{lemma}\label{lem:existenceprinc}If the algebraic multiplicity of the smallest eigenvalue of $H^g(p)$ is equal to one, then  $\grad^g F$ has exactly two principal flow lines at $p$.
\end{lemma}

\begin{proof}
Let $(\varphi,U)$ be the chart of Corollary~\ref{cor:diagsys}, and set $z =\varphi(x)$. The gradient ascent flow is then $$\frac{d}{dt} z = \Lambda z,$$ for $\Lambda = \diag(\lambda_1,\ldots,\lambda_n)$ and $\lambda_1 < \lambda_2 \leq \cdots \leq \lambda_n <0$. 
Let $r>0$ be so that $S_r(0) \subset V':=\varphi(U)$. Note that $S_r(0)$ parametrizes the set of gradient ascent flow lines that reach $p$; indeed, every such  flow lines intersects $S_r(0)$ at a unique $z_0 \in S_r(0)$, and is thus of the form $z(t) = \exp(\Lambda t) z_0$.

We can write $z_0 = \sum_{i=1}^n \zeta_{i} e_i$ for some coefficients $\zeta_{i} \in \R$, and $\exp(\Lambda t)= \sum_{i=1}^n e^{\lambda_it} e_ie_i^\top$.
 Since $e_i^\top e_j = \delta_{ij}$, where $\delta_{ij}$ is the Kronecker delta, we have that $z(t) = \sum_{i=1}^n e^{\lambda_i t} \zeta_ie_i $ and thus $$\Lambda z(t) = \sum_{i=1}^n  \lambda_i\zeta_{i}  e^{\lambda_i t}e_i=e^{\lambda_1 t}\left(\lambda_1\zeta_{1} e_1 +\sum_{i=2}^n \lambda_i\zeta_{i} e^{(\lambda_i-\lambda_1)t} e_i \right).
$$ The norm of the above vector is $$\|\Lambda z(t)\|=\left(\sum_{i=1}^n \lambda_i^2\zeta_{i}^2 e^{2\lambda_i t} \right)^{1/2}=e^{\lambda_1 t}\left(\lambda_1^2\zeta_{1}^2 +\sum_{i=2}^n \lambda_i^2\zeta_{i}^2 e^{2(\lambda_i-\lambda_1) t} \right)^{1/2}.$$

From the above two equations, we conclude that $$\lim_{t \to \infty} \frac{\Lambda z(t)}{\|\Lambda z(t)\|} = \lim_{t \to \infty}\frac{\lambda_1\zeta_{1} e_1 +\sum_{i=2}^n \lambda_i\zeta_{i} e^{(\lambda_i-\lambda_1)t} e_i}{\left(\lambda_1^2\zeta_{1}^2 +\sum_{i=2}^n \lambda_i^2\zeta_{2}^2 e^{2(\lambda_i-\lambda_1) t} \right)^{1/2}}.$$ 
Recall that by assumption, $\lambda_i-\lambda_1 >0$, $2 \leq i \leq n$. Since the $e_i$ are linearly independent, we conclude that the above limit is $\pm e_1$ if and only if $\zeta_i=0$ for $2 \leq i \leq n$, and thus $\zeta_1 = \pm r$. This concludes the proof, with the vector $v \in T_pM$ obtained by tracing back the changes of variable used.
\end{proof}

If the conditions of the Lemma are not met, a maximum of a Morse function can have more than two principal flow lines. For example,  consider   $F=-x^\top Q x$ on $\R^n$, where $Q$ is a positive definite matrix. Then $F$ has a maximum at the origin. If $Q= I$, then every flow line is a principal flow line. 

\begin{remark}[Intrinsic definition of principal flow lines]\label{rem:cL}In view of Lemma~\ref{lem:existenceprinc}, we can define the tangent vector to a principal flow line $v \in T_pM$ intrisically as follows. For vector fields $X, Y$, denote by $\cL_XY$ the Lie derivative of $Y$ along $X$. If $p$ is a zero of $X$, i.e., $X(p)=0$, then $(\cL_X Y)(p)$ depends on the value of $Y$ at $p$ only. Hence, we conclude that if $p \in \crit F$, we can define the linear map $\cL_{\grad F}:T_pM \to T_p M: w \mapsto \cL_{\grad F} W$ where $W$ is any differentiable vector field with $W(p)=w$. Then a short calculation shows that $\cL_{\grad F}$ has $H_\varphi(p)$ as matrix representation in the coordinates $\varphi$. The principal flow lines at $p$ are thus the trajectories of the gradient ascent flow that reach $p$ tangentially to $v \in T_pM$, where $v$ is an eigenvector of $\cL_{\grad F}$ corresponding to the smallest eigenvalue.
\end{remark}

We will denote the {\bf principal flow lines} of $\grad^g F$ at $p$ by $\gamma^+_t(p,g)$ and $\gamma^-_t(p,g)$. Equipped with the above Lemma, we define gradient vector fields with simple maxima:

\begin{definition}[Gradient vector fields with simple maxima]Let $(M,g)$ be a Riemannian manifold and $F \in C^\infty(M)$ be Morse function with a maximum at $p$. We say that $p$ is a simple maximum of $\grad F$ if  $H^g(p)$ has a unique smallest eigenvalue and   its principal flow lines belong to the stable manifolds of some minima of $ \grad F$. If all the maxima of $\grad F$ are simple maxima, we say that $\grad F$ is simple.
\end{definition}
The above definition can be reformulated as follows. Let $p \in \crit_nF$ and fix a choice $v_p$ of eigenvector spanning the eigenspace of $H^g(p)$ corresponding to the  smallest  eigenvalue.  Then $\nabla F$ is simple if for some (and thus all) $t \in \R$, \begin{equation}\label{eq:defgentleinc}\bigcup_{p \in \crit_nF}  \{\gamma^+_t,\gamma^-_t\} \subset \bigcup_{q \in \crit_0 F} W^s(q).\end{equation}

\subsection{Proof of the main theorem for simple gradients}

\begin{figure}
\centering
\includegraphics[width=.45\columnwidth]{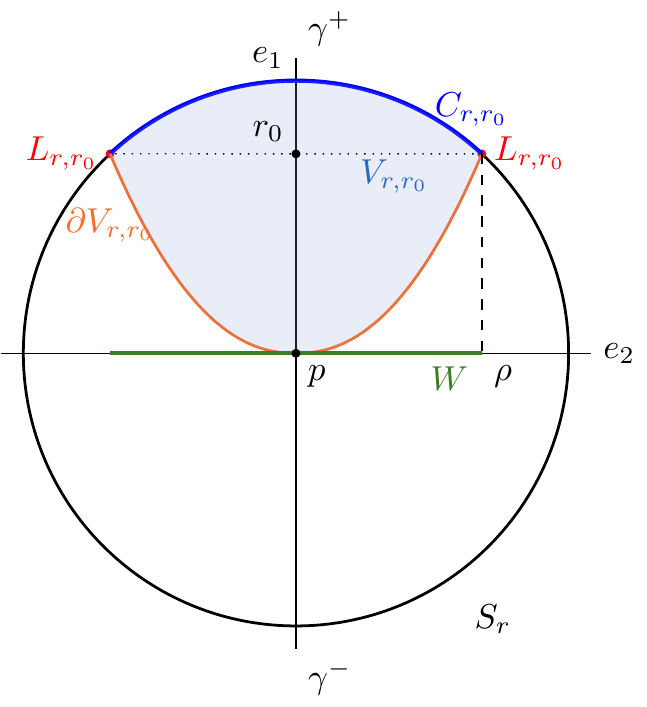}
\caption{\small The simple maximum $p$ of $\grad F$ for a two-dimensional $M$ has two principal flow lines, aligned with the $e_1$ axis. The sphere $S_r$ is centered at $p$ and $C_{r,x}$ is the spherical cap with the $e_1$ coordinate larger than $x$. The set $L_{r,x}$ is the boundary of $C_{r,x}$; it is a sphere of dimension $n-2$. Its image via the gradient flow is $\partial V_{r,x}$. We express, in Lemma~\ref{lem:domfunc}, the set $\partial V_{r.x}$ as a function from $W$ (here, the $e_2$ axis) to $\R$ (the $e_1$ axis.) }\label{fig:illcrx}	
\end{figure}

We now show that under the condition that $\grad F$ is simple, the inequality~\eqref{eq:maineq} holds.
We start with expressing an invariant set of $\nabla F$ as the epigraph of a differentiable function locally around a  maximum $p$. To describe this set, denote by $C_{r,r_0}$, for $0<r_0 < r$ the {\em top cap} of $S^{n-1}_r(0)$, where top cap refer to the first coordinate (i.e., along the $e_1$ axis) being greater than $r_0$ (see Fig.~\ref{fig:illcrx}).
Its boundary, which we denote by $L_{r,r_0}$, is a sphere of dimension $n-2$ centered at $r_0e_1$  given by:
\begin{equation}\label{eq:defL}L_{r,r_0}:=\{(z_1,z_2,\ldots,z_n) \mid z_1=r_0 \mbox{ and } \sum_{i=2}^n z_i^2 =\rho^2 \}=S^{n-2}_{\rho}(r_0e_1)\end{equation} where $\rho=\sqrt{r^2-r_0^2}$.
Let $\lambda_1 < \lambda_2 \leq \lambda_3 \leq \cdots \leq \lambda_n <0$,  $\Lambda =\diag(\lambda_1,\ldots,\lambda_n)\in \R^{n \times n}$ and define
the diagonal system
\begin{equation}\label{eq:sysdiag}
\dot z_i = \lambda_i z_i, \quad \mbox{ for } 1 \leq i \leq n.	
\end{equation} 
We let $V_{r,r_0}$ be the image of $C_{r,x}$ under the flow of Eq.~\eqref{eq:sysdiag}:
\begin{equation}\label{eq:defVRZ}
V_{r,r_0}:= e^{[0,\infty] \Lambda }\cdot C_{r,r_0} = \{z \in \R^n \mid z = e^{\Lambda t} y \mbox{ for } t \in [0,\infty], y \in C_{r,r_0} \}.	
\end{equation}
The boundary of $V_{r,r_0}$ is then $\partial V_{r,r_0} = e^{ \Lambda t}\cdot L_{r,r_0}$. The following result expresses this boundary as the graph of a function from $\R^{n-1} \to \R$, where  by convention the domain $\R^{n-1}$ is the space spanned by $\{e_2,\ldots,e_n\}$, and the codomain is spanned by $e_1$.
\begin{lemma}\label{lem:domfunc}
Let $V_{r,r_0}\subset \R^n$ be as in~\eqref{eq:defVRZ} for the dynamics of~\eqref{eq:sysdiag} and let $W:=\{w \in \R^{n-1} \mid \|w\| \leq \rho\}$. Then $\partial V_{r,r_0}$ is the graph of a positive differentiable function $F_l:W \to \R$, i.e., $\partial V_{r,r_0}=\{(F_l(w),w)\mid w \in W\}$. Furthermore, for $F_u:W \to \R: w\mapsto  \left(\frac{ \|w\|}{{\rho}}\right)^{\frac{\lambda_1}{\lambda_2} }r_0$, it holds that $$F_l(w) \leq F_u(w).$$ 	
\end{lemma}
	The case $n=2$ is proven: we have that $L_{r,r_0} =\{(r_0, \rho),(r_0,-\rho)\}$ and $\partial V_{r,r_0} = \{( e^{\lambda_1t}r_0,\pm e^{\lambda_2t}\rho) \mid t \in[0,\infty]\}$. 		 A short calculation yields that $\partial V_{r,r_0} = (F_l(w),w)$  for the function $$F_l:w \mapsto \left(\frac{|w|}{\rho}\right)^{\frac{\lambda_1}{\lambda_2}}r_0,\quad  w \in [-\rho,\rho].$$ We now prove the general case:

\begin{proof}[Proof of Lemma~\ref{lem:domfunc}]
	Denote a point in $\R^n$ as $z_1e_1+z_2e_2+\cdots+z_{n}e_n$ and recall the definition of $L_{r,r_0}$ in Eq.~\eqref{eq:defL}.

	Set $S_\rho^{n-2}:=\{(z_2,\ldots,z_n) \mid \sum_{i=2}^n z_i^2 =\rho^2 \}$. 
	From Eq.~\eqref{eq:sysdiag}, we obtain   
	$$\partial V_{r,r_0} =\{ (e^{\lambda_1 t}z_1,e^{\lambda_2 t}z_{2}, \ldots, e^{\lambda_n t}z_{n} )\mid (z_1,\ldots,z_n) \in L_{r,r_0}, t\in[0,\infty]\}.$$ 
Set $W_0:=W-\{0\}$. The map $$\Phi:[0,\infty) \times S^{n-2}_\rho \to W_0:(t,z_2,\ldots,z_n) \mapsto (e^{\lambda_2 t}z_2, \ldots,e^{\lambda_n t}z_n)$$ is a diffeomorphism onto its image.  Recalling that $\pi_1$ is the projection onto the first coordinate, we see that $\partial V_{r,r_0}-\{0\}$ is the graph of $$F_l(w):=\exp(\lambda_1 \pi_1(\Phi^{-1}(w))r_0,$$ which is differentiable and can be differentiably extended by $0$ at $0$.

 We now show that $F_u$ dominates $F_l$ over $W_0$. To see this, it is easier to work in the coordinates afforded by $\Phi^{-1}$: in these coordinates, $w = \Phi(t,z_2,\ldots,z_n)$ and, recalling that $\pi_{-1}$ is the projection $(z_1,z_2,\ldots,z_n)\mapsto(z_2,\ldots,z_n)$, we have 
\begin{align*}
F_u(w)=F_u(\pi_{-1}(e^{\Lambda t}z))&=	\left(\frac{\sqrt{\sum_{i=2}^n e^{2\lambda_it}z_i^2} }{\rho}\right)^{\lambda_1/\lambda_2}r_0\\
&=\left(e^{\lambda_2t}\frac{\sqrt{z_2^2+\sum_{i=3}^n e^{2(\lambda_i-\lambda_2)t}z_i^2} }{\rho}\right)^{\lambda_1/\lambda_2}r_0\\
&\geq e^{\lambda_1t}r_0=F_l(w)
\end{align*}
where we used the facts that $\lambda_1<\lambda_2 \leq \lambda_n<0$, $3 \leq i \leq n$ and $\sum_{i=2}^n z_i^2 = \rho^2$ to obtain the inequality.
\end{proof}

We are now ready to prove that inequality~\eqref{eq:maineq} holds for simple gradient flows. 

\begin{proposition}\label{prop:main1} Let $M$ be a closed manifold, and $\mu$ and $d$ as in Theorem~\ref{th:main}. Let $(F,g)$ be so that $\nabla^g F$ is simple, Then
	for $p \in \crit_n(F)$, there exists $m_+(p),m_-(p) \in \crit_0(F)$, not necessarily distinct, with the property that for all $\varepsilon>0$, there is $\delta>0$  such that 
	$$\mu\left(x \in B_\delta(p) \mid \lim_{t \to \infty}e^{-t \grad F}x \in \{m_+(p),m_-(p)\}\right)	 \geq (1-\varepsilon) \mu(B_\delta(p)).$$
\end{proposition}
 
\begin{proof}
Fix $\varepsilon >0$. Let $(\varphi,U)$ be a chart as in Corollary~\ref{cor:diagsys}. The gradient ascent flow in the coordinates given by $z=\varphi(x)$ has the form $$\dot z =\Lambda z$$ in $\varphi(U)$, where $\Lambda = \diag(\lambda_1,\ldots,\lambda_n)$. The principal flow lines $\gamma^+$ and $\gamma^-$ of $p$ are aligned with the half-lines $\{z_1e_1\mid z_1>0\}$ and $\{z_1e_1\mid z_1<0\}$, respectively. 

Because $p$ is simple, the principal flow lines $\gamma^+$ and $\gamma^-$ of $p$ belong to the stable manifold of some  minima of $F$; denote them $m_+(p), m_-(p) \in \crit_0 F$ respectively. 
Let $K \subset \psi(U)$ be a compact, contractible set containing the origin in its interior.  Since the distance $d$ is Riemannian, it is uniformly comparable to the Euclidean distance in $K$, i.e., there exists constants $\beta>\alpha>0$ such that \begin{equation}\label{eq:uniriem}\alpha \|z\| \leq d(0,z) \leq \beta \|z\|,  \mbox{ for all }z \in K, \end{equation} where with a slight abuse of notation, we write $d(0,z)$ for $d(p,\varphi^{-1}(z))$. Fix $r >0$ such that $S_{r/\alpha} \subset K$ and $S_r \subset K$. 
 For any $0<\delta <r$, let $B_\delta:=\{z \mid d(0,z) \leq \delta\}$ be the ball of radius $\delta$ centered at $0$ for the distance $d$, and by $D_\delta:=\{z \mid \|z\| \leq \delta\}$ the ball of radius $\delta$ centered at $0$ for the Euclidean distance. We also let $B^+_\delta:=\{z \in B_\delta \mid z_1 \geq 0\}$ (and $B^-_\delta$ is defined in the obvious way), and define the half-balls $D^\pm_\delta$  for the Euclidean distance similarly.

Because $p$ is simple, we have  $r e_1 \in W^s(m_+(p))$,  and because  $W^s(m_+(p))$ is open in $M$,  there exists $r_0 \in (0,r)$ such that the closed spherical cap $C_{r,r_0}$ of $S_{r}$ is contained in $W^s(m_+(p))$; see Fig.~\ref{fig:illpth}-left. 
Hence $V_{r,r_0} \subseteq W^s(m_+(p))$, where we recall that $V_{r,r_0}$ is the image of $C_{r,r_0}$ under the flow as we defined in~\eqref{eq:defVRZ}.  

We claim  that 
\begin{equation}\label{eq:th1inter}
\lim_{\delta \to 0} \frac{\mu(B^+_\delta \cap  V_{r,r_0})}{\mu(B^+_\delta)}=1\end{equation}
 and similarly, that $  \lim_{\delta \to 0} \frac{\mu(B^+_\delta \cap  V^-_{r,r_0})}{\mu(B^+_\delta)} =1$, where $V^-_{r,r_0}$ is the image of a lower spherical cap under the flow. Assuming the claim holds, using elementary properties of measures,  we have that (see Lemma~\ref{lem:measure} in the Appendix for a proof) 
 $$ \lim_{\delta \to 0}\frac{\mu(B_\delta \cap (V_{r,r_0}\cup V^-_{r,r_0}))}{\mu(B_\delta)} = 1$$
 
 Since  $\left(V_{r,r_0}\cup V^-_{r,r_0}\right) \subseteq \left(W^s(m_+(p)) \cup W^s(m_-(p))\right) $, we conclude that for all $\varepsilon>0$, there exists $\delta$ so that \begin{equation}\label{eq:theconc1}
 	\mu\left( \left[W^s(m_-(p))\cup W^s(m_+(p))\right]\cap B_{\delta}\right)\geq (1-\varepsilon)\mu(B_{\delta}),
 \end{equation} 
as announced.

 \begin{figure}[t]
 \includegraphics[scale=.8]{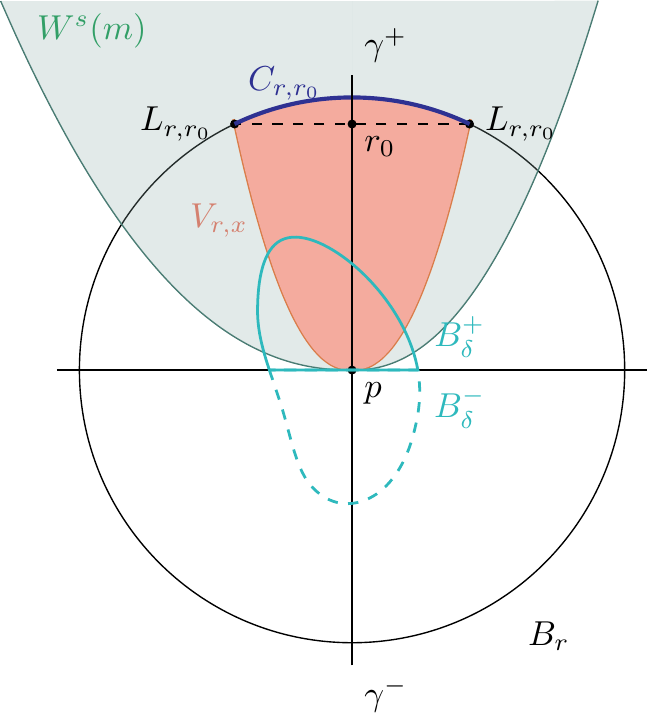}\quad	
 \includegraphics[width = 3.2in, height = 2.6in]{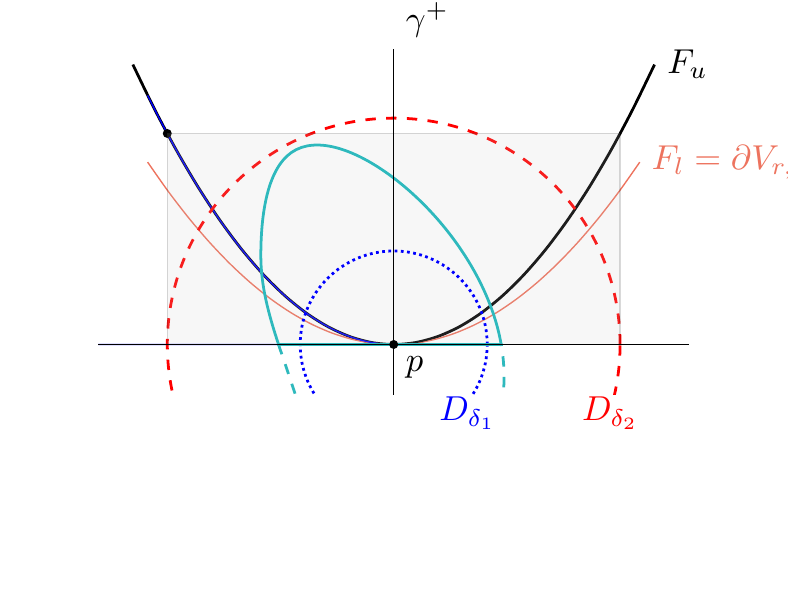} 
 \caption{\small {\it Left: } In the coordinates of Cor.~\ref{cor:diagsys}, the local principal flow lines are the positive and negative $e_1$ (vertical) axis. The spherical cap $C_{r,x}$ is contained in a stable manifold, thus so is its image $V_{r,x}$  under the gradient ascent flow. {\it Right: } We can express $V_{r,x}$ as the epigraph of $F_l(w)$, which is dominated by $F_u(w)$ and thus $\hyp(F_l) \subset \hyp F_u$. Furthermore, $ D_\delta^+ \cap \hyp(F_u)$ is contained in $B_{\delta_2} \cap \hyp (F_u)$ which is itself contained into the cylinder with base a ball of radius $\delta_2$ and height $F_u(\delta_2)$ (light shaded rectangle).}\label{fig:illpth}
 \end{figure}

It now remains to prove the claim, i.e.~prove that~\eqref{eq:th1inter} holds.  Let $W$ and $F_u(w), F_l(w)$ be as in Lemma~\ref{lem:domfunc} and define the graph of $F:W \to \R$ as the set  $\{(F(w),w) \in \R^n\mid w \in W\}$. We denote by $\epi(f)$ the epigraph of a function $f$, and by $\hyp(f)$ its hypograph. Since $F_u \geq F_l$, we have  that (see Fig.~\ref{fig:illpth}-right) 
 $$\epi(F_u)\cap D_\delta^+\subseteq \epi(F_l) \cap D_\delta^+= V_{r,r_0}\cap D_\delta^+,$$ 
for any $0 < \delta<r$.	
Passing to hypographs, we have
\begin{equation}\label{eq:inctech}
V_{r,r_0}\cap D^+_\delta = D_{\delta}^+-(D_\delta^+ \cap \hyp(F_l)) \supseteq  D_{\delta}^+-(D_\delta^+ \cap \hyp(F_u)).
\end{equation}
From~\eqref{eq:uniriem}, we have the inclusions 
\begin{equation}\label{eq:techthincB}D_{\delta_1} \subseteq B_\delta \subseteq D_{\delta_2}
\end{equation}
 for $\delta_1:=\frac{\delta}{\beta}$ and $\delta_2:=\frac{\delta}{\alpha}$. Hence,
 \begin{equation}\label{eq:techthm1} 
B_\delta^+ \cap \hyp(F_u)\subseteq D^+_{\delta_2}\cap \hyp(F_u) . 
\end{equation} 
From~\eqref{eq:techthincB} and~\eqref{eq:techthm1}, we have that\begin{equation}\label{eq:techthm2}
\frac{\mu\left(B_\delta^+ \cap \hyp(F_u)\right)}{\mu(B_\delta^+)} \leq 	\frac{\mu\left(D_{\delta_2}^+ \cap \hyp(F_u)\right)}{\mu(D_{\delta_1}^+)}.
\end{equation}

Because $F_u(w)$ is rotationally symmetric about $e_1$ and strictly increasing as $\|w\|$ increases, we have $$\mu(D_{\delta_2}^+ \cap \hyp(F_u)) \leq c \delta_2^{n-1}F_u({\delta_2}) \leq c \delta^{n-1+\lambda_1/\lambda_2},$$ whereas $\mu\left(D_{\delta_1}^+\right) = c \delta^n$.
Since $\lambda_1/\lambda_2>1$, we conclude from the previous relation together with~\eqref{eq:techthm2} that \begin{equation}\label{eq:techth3}0\leq \lim_{\delta \to 0} \frac{\mu\left(B_\delta^+ \cap \hyp(F_u)\right)}{\mu(B_\delta^+)}  \leq \lim_{\delta \to 0}\frac{\mu(D_{\delta_2}^+ \cap \hyp(F_u))}{\mu(D_{\delta_1}^+)}=0.\end{equation}
From~\eqref{eq:inctech}, we have that 
\begin{equation}\frac{\mu\left(V_{r,r_0}\cap B^+_\rho\right)}{\mu(B_\delta^+)} \geq 1-\frac{ \mu\left(B_\delta^+ \cap \hyp(F_u)\right)}{\mu(B^+_\delta)}.
\end{equation}	
Taking the limit as $\delta \to 0$, using~\eqref{eq:techth3} and recalling that $V_{r,r_0} \subseteq W^s(m^+(p))$ we get that $$\lim_{\delta \to 0}\frac{\mu\left(W^s(m^+(p))\cap B^+_\delta\right)}{\mu(B_\delta^+)} =1,$$ thus proving~\eqref{eq:th1inter} as claimed.
Applying the same reasoning to the stable manifold of $m^-(p)$ and $B^-_\delta$, we get that  similarly  $\lim_{\delta \to 0}\frac{\mu\left(W^s(m^-(p))\cap B^-_\delta\right)}{\mu(B_\delta^-)} =1.$  \end{proof}

 \subsection{Simple gradients are generic}
 
 We now prove that gradient vector fields with simple maxima are generic. There are two requirements to being simple: (1) the smallest eigenvalue of the linearized vector field has geometric multiplicity one, and (2) the principal flow lines need to be contained in stable manifolds of minima of $F$. We treat the two requirements separately. 

To this end, for $F \in C^\infty(M)$, we denote by $\cM_{0,F}$ the set of Riemannian metrics $g$ on $M$ with the property that the smallest eigenvalue of the linearization of $\nabla^g F(p)$  has  geometric multiplicity one when   evaluated at any maximum $p \in \crit_n(F)$. We  write the requirement simply as $\lambda_1(H^g(p))<\lambda_2(H^g(p))$ for all $p \in \crit_nF$. We further denote by $\cM_F$ the subset of $\cM_{0,F}$ consisting of metrics $g$ for which $\grad^g F$ has {\it simple} maxima. Given $g \in \cM$,  we similarly let $\cF_{0,g}$ be the set of Morse functions $F \in C^\infty(M)$ on $(M,g)$ so that for each $p \in \crit_n F$, $\lambda_1(H^g(p)) <\lambda_2(H^g(p))$ and $\cF_g$ the subset of $\cF_{0,g}$ consisting of functions $F$ for which $\grad^g F$ has simple maxima. We will show that $\cM_F$ is residual in $\cM$ and that $\cF_g$ is residual in $C^{\infty}(M)$.

\subsubsection{Geometric multiplicity of the smallest eigenvalue}

We prove that the set of metrics  for which the linearization of $\grad^g F$ has a smallest eigenvalue of multiplicity one at each maximum $p$ is open-dense:

\begin{proposition}\label{prop:cg0F}
The set $\cM_{0,F}$ is open and dense in $\cM$.	
\end{proposition}
\begin{proof}	
We first show the set is open. Let $F$ be a Morse function so that for each $p \in \crit_nF$, 	$\lambda_1(H^g(p))<\lambda_2(H^g(p))$. Since the eigenvalues of $H^g(p)$ depend continuously on $g$, there exists an open set $\cU_p \subset \cG$ so that for all $h \in \cU_p$, $\lambda_1(H^h(p))<\lambda_2(H^h(p))$. Since $|\crit_n F|$ is finite, $\cU:= \bigcap_{p \in \crit_nF} \cU_p \subset \cM$ is an open set containing $g$. Hence $\cM_{0,F}$ is open.

To show that $\cM_{0,F}$ is dense, assume that $g$ is so that there exists $p \in \crit_nF$ with $\lambda_1(H^g(p))=\lambda_2(H^g(p))$. We show that we can find, in any open set containing $g$, a metric $h$ so that $\lambda_1(H^h(p))<\lambda_2(H^h(p))$. Recall that in coordinates around $p$ sending $p$ to $0 \in \R^n$, we can write $H^h(p)= H^{-1}(0)\frac{\partial^2 F}{\partial z^2}$, where $H(x)$ is a positive definite matrix defined in a neighborhood of $0$. Using a bump function around $p$, the fact that the map $X \mapsto X^{-1}$ is a diffeomorphism around $X=H^{-1}(0)$, and Lemma~\ref{lem:techmatrix1} (which states that if a product $AB$ of two positive definite matrices has repeated eigenvalues, there exists $A'$ {\em positive definite} and arbitrarily close to $A$ so that  $A'B$ has {\it distinct} eigenvalues), we can obtain a metric $h$ arbitrarily close to $g$ and so that $H^{-1}(0)\frac{\partial F}{\partial x^2}$ has distinct eigenvalues.
\end{proof}

We now show the equivalent result for a fixed metric $g$ and arbitrary Morse function $F$. Just as above, we in fact prove the stronger statement that the set of Morse function so that $H^g(p)$ has {distinct} eigenvalues at each of the critical points of $F$ is open dense.  The proof relies on the notion of jet tranversality -- we refer to~\cite{hirsch2012differential} for an introduction.
\begin{proposition}\label{prop:cf0}
	The set $\cF_{0,g}$ is open dense in $C^\infty(M)$.
\end{proposition}

\begin{proof}
We know that Morse functions are an open dense subset of $C^\infty(M)$~\cite{banyaga2013lectures}. We show that Morse functions for which $H^g(p)$ has distinct eigenvalues at $p \in \crit_n F$ form an open dense subset of the set of Morse functions, and thus are open dense in $C^\infty(M)$.

Given $A \in \R^{n \times n}$, denote by $p_A(s)$ its characteristic polynomial in the indeterminate $s$ and let $p'_A(s)=\frac{d}{ds}p_A(s)$. Denote by $r_A \in \R^{2n\times 2n}$ the Sylvester resultant of $p_A$ and $p'_A$. It is well known that $\det(r_A)=0$ if and only if $p_A$ has a double root. 
Let $Z\subset \R^{n \times n}$ be the zero set of $\det(r_A)$. Relying on Whitney's stratification theorem~\cite{whitney1957elementary}, we can show that $Z$ is a finite union of closed manifolds.

Denote by $J^2(M,\R)$ the second jet-space of maps $F:M\to \R$ and define $$C =\{(x,y,0_n,H)\in J^2(M,\R) \mid x\in M, y \in \R, G^{-1}(x)H \in Z)\},$$ where $G(x)$ is the matrix expression of $g$. Then $C$ is a finite union of submanifolds of $J^2(M,\R)$ of codimension $\geq n+1$ (since we restrict the first derivative to be zero, and $Z$ is the union of submanifolds of codimension at least one.)	 Consequently, the second jet prolongation of $F$, $j^2(F)$, and $C$ are transversal {\it only} at points at which they do not intersect. Furthermore, $C$ is easily seen to be closed in $J^2(M,\R)$.  Hence, from the jet-transversality theorem~\cite{hirsch2012differential}, we conclude that the set of real-valued functions without critical points for which $H^g(p)$ has repeated eigenvalues is open and dense in $C^\infty(M)$.
\end{proof}

\subsubsection{Continuity of principal flow lines with respect to $F/g$} We now address the second part of the simplicity of $\grad F$ requirement: the principal flow lines of each maxima belong to the stable manifolds of minima of $F$. The first step is to establish that principal flow lines depend continuously on the metric/function.

\begin{lemma}\label{lem:contprincipal} Let $(M,g)$ be a closed Riemannian manifold. Let $F$ be a smooth Morse function, and $p \in M$ a simple maximum of $\grad^g F$. Then, there exists a $C^1$-embedded closed ball $B_p \ni p$ in $M$ and an open set $\cU \subset \cM$ containing $g$  with the following properties:
\begin{enumerate}
\item $B_p$ contains no other critical points of $F$

\item  the principal flow line $\gamma^+(p,h)$ (resp. $\gamma^-(p,h)$) intersect $\partial B_p$ at one point, and the intersection $\gamma^+ \cap \partial B_p$ (resp. $\gamma^- \cap \partial B_p$) depends continuously on $h$, $h \in \cU$. 
\item the boundary $\partial B_p$ is everywhere transversal to $\grad^{h} F$, $h \in \cU$ 
\item $B_p$ is an invariant set for the gradient ascent flow of $\grad^{h} F$, $h \in \cU$.
\end{enumerate}
\end{lemma}

\begin{proof}
 We work in the chart $(\varphi,U)$ afforded by Corollary~\ref{cor:diagsys} sending $p$ to $0 \in \R^n$, and for which the gradient flow differential equation is  $\dot z = \Lambda z$,  with $\Lambda$ a diagonal matrix with diagonal entries $\lambda_1<\lambda_2 \leq \cdots \leq \lambda_n<0$.
 
 Since $\grad^h F$ depends continuously on $h$, from the proof of Hartman's theorem~\cite{hartman1960mex}, we know that there exists a neighborhood $V \ni 0$, a neighborhood $\cU_0 \subset \cM$ of $g$ and a {\it continuous} mapping $\psi:\cU_0 \to \Diff(V,\R^n)$ such that for any metric $h \in \cU_0$, the diffeomorphism $\psi_h: V \to \R^n$ linearizes $\grad^h F$ around $0$ (see also~\cite[p. 215]{newhouse2017}, the author calls the continuous dependence of the linearizing diffeomorphism with respect to the vector field {\it robust linearization}). 
Note that in the coordinates used, $\psi_g=Id$.
 
The principal flow lines of $\grad^gF$ in the $z$-coordinates  are locally given by the half-lines starting at the origin and spanned by the vectors $\pm e_1$.  Let $r>0$ be such that $B_r:=B_r(0) \subset V$. The half-lines intersect $\partial B_r=S_r$ at exactly two points, denote them $z_+(g), z_-(g)$, and these intersections are clearly transversal.

Taking a subset $\cU_1 \subset \cU_0$, we can ensure that for all $h \in \cU_1$, $\lambda_1(H^h(p))<\lambda_2(H^h(p))$, since the eigenvalues depend continuously on $h$.  Similarly, in the (linearizing) coordinates $\psi_h$, the principal flow lines of $\grad^h F$ are half-lines starting at the origin and spanned by an eigenvector $v_1(H^h_\psi(p))$ associated with $\lambda_1(H^h_\psi(p))$ and, from Lemma~\ref{lem:conteig}, we know that the eigenspace $v_1(H^h_\psi(p))$ depends continuously on $h$ as well. 
The principal flow lines of $\grad^hF$ in the $z$-coordinates  are given by the image under $\psi_h^{-1}$ of the half-line  starting at zero and parallel to $v_1(H^h_\psi(p))$, and thus depend continuously on $h$. Now since the principal flow lines of $\grad^gF$ intersect $S_r$ transversally, by taking a subset $\cU_2 \subset \cU_1$, we can ensure that for all $h \in \cU_2$, the principal flow lines of $\grad^h F$ in $z$-coordinates intersect $S_r$ transversally and the intersections $z_+(h), z_-(h)$ are continuous in $h$.

 Finally, for the last two items, since $\grad^gF$ is linearized by $\varphi$ as $\dot z = \Lambda z$, with $\Lambda$ diagonal and with negative, real eigenvalues, then $\grad^g F$ evaluated on $S_r$ points inward, toward $B_r$: indeed, the inward pointing normal to $S_r$ at $z$ is $-z$ and  its inner product with $\grad^gF$ is $-z^\top \Lambda z >0$. Because $S_r$ is compact, the same conclusion holds for vector fields close enough to $\grad^g F$. Hence $B_r$ is invariant for $\grad^hF$, for $h$ close to $g$. 
 Setting $B_p$ to be the inverse image under the chart  \begin{equation}\label{eq:defBp}B_p:=\varphi^{-1}(B_r(0)),\end{equation} we obtain a set with the required properties.
\end{proof}
 
\begin{remark}
The above result transposes immediately to the case where the Riemannian metric $g$ is fixed, and we consider an open set of function $\cU_0 \subset C^\infty(M)$ containing $F$ where $\grad^gF$ has a simple maximum at $p$. The continuous dependence of $\grad^gF$ on $F$ is obvious. The only point of demarcation is that when varying $F$ to a nearby $F_1$, the critical points of $\grad^gF_1$ may move. It is easy to see though that for a $\cU_0$ small enough, they move continuously and their index remains the same: there exists a	{\it continuous} map  $P:\cU_0 \to V\subset M$ so that $P(F_1)$ is a critical point of $F_1$. (See, e.g.,\cite[Lemma 3.2.1]{palis1982geometric} or~\cite{newhouse2017}).
\end{remark}

\subsubsection{Genericity of simple gradients}
We now prove the second part of the main theorem, namely that simple gradient flows are generic. 

\begin{proposition}\label{prop:main2}
Let $F \in C^\infty(M)$ be a Morse function. The set $\cM_F$ of Riemannian metrics for which $\grad^g F$ is simple is residual. Similarly, for a Riemannian manifold $(M,g)$, the set $\cF$ of smooth functions for which $\grad^g F$ is simple is residual.
\end{proposition}

We prove the first statement, and then indicate the minor changes needed to obtain the second statement.

\begin{proof}

Pick a Morse function $F \in C^\infty(M)$ and metric $g \in \cM_{0,F}$. We denote by $p_1,\ldots, p_{m}$ and by $s_1,\ldots,s_l$ the maxima and  saddle points of $F$, respectively. We have shown that $\cM_{0,F}$ is open dense in $\cM$, it  thus remains to show that metrics $g$ in $\cM_{0,F}$ for which the principal flow lines of $\grad^gF$ at $p_i$, $1 \leq i \leq m$, belong to the stable manifolds of some minima form a generic set. Owing to the stable manifold decomposition of $M$ and the fact that $W^s(p)=\{p\}$ for $p \in \crit_n F$, it is equivalent to show that, generically for $g$, the principal flow lines of $\grad^g F$ at $p_i$ do not belong to the stable manifold of some saddle points.

To this end, we will make use of the following straightforward characterization of generic sets:  given that  $\cM_{0,F}$ is dense in $\cM$, the subset $\cM_{F}\subseteq \cM$ is generic if and only if for each $g \in \cM_{0,F}$, there exists a neighborhood $\cN_g$ of $g$ in $\cM$ so that $\cM_{F}  \bigcap \cN_g$ is generic in $\cN_g$. For a proof of this statement, we refer to, e.g.,\cite[Lemma 3.3.3]{palis1982geometric}. The statement allows us to consider only elements of $\cM_{0,F}$, which is an easier task than considering any element of $\cM_F$.

 For each $s_i$, we let $W_0^s(s_i,g)$ be a compact neighborhood of $s_i$ in the stable manifold $W^s(s_i,g)$. Let $\Sigma_i^s$ be a codimension one submanifold of $M$ that is (1) transversal to $\grad^g F$ and to $W_0^s(s_i,g)$ and (2) meets $W_0^s(s_i,g)$ at the boundary $\partial W_0^s(s_i,g)$. The construction of the set $\Sigma_i^s$ appears in the proof of the Kupka-Smale theorem~\cite{peixoto1967approximation}, and we refer to, e.g.,~\cite[p.107]{palis1982geometric} for a constructive proof of its existence.

Because $\grad^h F$ depends continuously on $h$, we know from the stable manifold theorem~\cite[Th.~2.6.2]{palis1982geometric} that for $h$ in a small enough neighborhood $\cN_g \subset \cM_{0,F}$,  the maps $h \mapsto W_0^s(s_i,h)$, $1 \leq i \leq l$, are continuous and so that $W_0^s(s_i,h)$ intersects $\Sigma_i^s$ transversally at $\partial W_0^s(s_i,h)$, $1 \leq i \leq l$. Note that since $\cM_{0,F}$ is open in $\cM$, $\cN_g$  is also a neighborhood of $g$ in $\cM$.
	
 Let $k \geq 1$ be a positive integer. Define $$W_k^s(s_i,h):=e^{-k\grad^hF}\cdot W_0(s_i,h),$$ i.e., the image of $W_0^s(s_i,h)$ by applying the gradient flow  for a time of $k$ (or the gradient {\it ascent} flow for a time $-k$.) Since $e^{-k\grad^hF}: M \to M$ is a diffeomorphism for each $k$, $W_k^s(s_i,h)$ is a compact subset of $M$ that depends continuously on $h$. Finally, we have by definition that $$W^s(s_i,h) = \bigcup_{k\geq 0} W_k^s(s_i,h).$$

Let $\cM_{k,i}(\cN_g)\subseteq \cM_{0,F} \bigcap \cN_g$ be the set of metrics $h$ in $\cN_g$ for which the local principal flow lines of $\grad^h F$ at $p_i$ do {\em not} intersect $W^s_k(s_j,h)$ for all $1 \leq j \leq l$, $1 \leq i \leq m$. Let $$\cM_k(\cN_g) =\bigcap_{i=1}^m \cM_{k,i}(\cN_g).$$  We will show that for all $k \geq 0$, $\cM_k(\cN_g) $ is open and dense in $\cN_g$. Since $\cap_{k=1}^\infty \cM_k(\cN_g) = \cM_F \cap \cN_g$,  this shows that $\cM_F \cap \cN_g$ is generic and, using the characterization of generic sets described above, proves the result. 
	
{\em $\cM_{k,i}(\cN_g)$ is open in $\cN_g$:}
We show that for any $h \in \cM_{k,i}(\cN_g)$, there exists an open neighborhood $\cU_h$ of $h$ contained in $\cM_{k,i}(\cN_g)$.
		 
To this end, let $B_{p_i} \subset M$ and $\cU^i \subset \cM_{0,F}$ be the closed ball and open set, respectively, from Lemma~\ref{lem:contprincipal} for the metric $h$. Since $\grad^hF$ is transversal to $\partial B_{p}$ and $\codim \partial B_{p_i} =1$, then $ W_k(s_j,h)$ and  $\partial B_{p_i}$ intersect transversally. Additionally, because the map $h' \mapsto W^s_i(s_j,h')$ is continuous for $h' \in \cM_{k,i}(\cN_g)$, so are the intersections of $W^s_i(s_j,h')$ with $B_{p_i}$ as a function of $h'$. 	From the same Lemma, denoting by $\gamma_{i,0}^{h'}$ the (positive) local principal flow line of $\grad^{h'}F$ at $p_i$, we know that the map $h' \mapsto \gamma_{i,0}^{h'} \cap \partial B_{p_i}$ is continuous as well. 
	
Putting the above two facts together, we conclude that there exists a neighborhood $\cU_h$ of $h$ in $\cM_{k,i}(\cN_g) $  so that for all $h'\in \cU_h$,  the principal flow lines $\gamma_{i,0}^{h'}$ do not intersect $W_k^s(s_i,h')$. Hence $\cM_{k,i}(\cN_g)$ is open.

\vspace{.2cm}
{\em $\cM_{k,i}(\cN_g)$ is dense in $\cN_g$:}
We will show that for any element $h \in \cN_g$, there exists an element $h' \in \cM_{k,i}$ arbitrarily close to $h$.
If $h \in \cM_{k,i}$, there is nothing to prove. Hence assume, to fix ideas, that the positive principal flow line $\gamma^h_{i,0}$ intersects $\cup_{j=1}^l W_n^s(s_j,h)$. 
	
 Using the local change of variables $\varphi_h: (U,p_i) \to (\R^n,0)$ afforded by Hartman's 
 Theorem (Theorem~\ref{th:hartman}) around the maximum $p_i$, the system follows the dynamics $\dot z =\Lambda z$ and after potentially another linear change of variables, we can assume that $\Lambda=\diag(\lambda_1,\ldots,\lambda_n)$, with $\lambda_1<\lambda_2 \leq \cdots \leq \lambda_n <0$ the eigenvalues of $H^h(p)$. From Eq.~\eqref{eq:defBp}, we know that $B_{p_i}$ in the $z$ coordinates is a ball $B_r(0)$ of given radius $r>0$ and centered at $0$.
	
 Denote by $\gamma^+_0$ the segment $(te_1,0,\ldots,0) \in \R^n$, $0 \leq t\leq r$. It is a compact subset of the (positive) principal flow line of $\grad^hF$ at $p_i$.  Let $z_0 = (r/2,0,\ldots,0)$. Since $z_0$ is not a critical point of $\grad^hF$, by the flowbox theorem~\cite[p.~93]{palis1982geometric}, we know there exists a neighborhood $U_0$ of $z_0$, which we take to be included in the ball of radius $r/2$ around $z_0$, and a local diffeomorphism $\varphi_0:U_0 \to \R^n$ under which the dynamics is, in the new variables induced by $\varphi_0$ (which we denote by $y$) given by 
	 \begin{equation}\label{eq:sysyflowbox}\dot y= (1,0,\ldots, 0).\end{equation} 
	 Without loss of generality, we can assume that $\varphi_0(z_0)=(r/2,0,\ldots,0)=:y_0$. See Fig.~\ref{fig:flowK} for an illustration.
	
	Working in the $y$-coordinates, let $K$ be a box (unit ball for $\| \cdot \|_\infty$ norm) centered at $y_0$ and of width $0<r'<r/2$ small enough so that $\varphi_0^{-1}(K)\subset B_r$.  
		  For any $\tilde h \in \cM$ which agrees with $h$ outside of $K$, because $\grad^hF$ and $\grad^{\tilde h}F$ then also agree outside of $M-B_r$ and this set is invariant under the flow $-\grad^h F$ by Lemma~\ref{lem:contprincipal}, we have that  \begin{equation}\label{eq:hhbarWs}W^s_k(s_j,\tilde h )\cap \partial B_r =W^s_k(s_j,h)\cap \partial B_r.\end{equation}
	  
 Let $y_1:=y_0-(r'/2,0,\ldots,0) \in \partial K$. Let $\theta:M \to \R$ be a smooth positive function with support $K$ and such that $$\int_{0}^{r'} \theta(y_1+te_1) dt=1.$$ Now define the following smooth vector field with support in $K$: for $v \in \R^n$, $$Y_v(y) := \theta(y)v.$$
Let $\phi^{v}_t(y_1)$ be the solution at time $t$ of the Cauchy problem \begin{equation}\label{eq:caucy0}\dot y =  \grad^hF(y)+Y_v(y), y(0)=y_1,\end{equation}

To proceed, we show that we can always find a metric for which the vector field in Eq.~\eqref{eq:caucy0} is the gradient of $F$:

\begin{lemma}\label{lem:yvmetricgrad}
	For $\delta>0$ small enough, there exists a metric-valued function $h_v$ for all $v\in \R^n$ with $\|v\| < \delta$, depending continuously on $v$, agreeing with $h$ outside of $K$, so that $$\grad^{h_v}F(y) = \grad^h F(y)+ Y_v(y).$$
\end{lemma}
\begin{proof}

Because $K$ does not contain any critical points of $F$, we have that  $dF\cdot \grad^h F >0$. Thus, for $\delta$ small enough,  we have that $dF \cdot (\grad^h F+Y_v) >0$ for all $v$ with $\|v\|<\delta$, $y \in K$. Set $Z_v:=\grad^h F+Y_v$.
 
From the above, we can decompose the tangent space $T_y M = \Span Z_v(y) \oplus \ker dF(y)$ for $y \in K$. We now introduce a metric for which this decomposition of the tangent space is orthogonal. In coordinates, it has the matrix expression $$h_v:=\begin{pmatrix} dF\cdot Z_v & 0 \\0 & h_{|\ker dF} \end{pmatrix},$$ where $h_{|\ker dF}$ is the restriction of $h$ to the $n-1$ dimensional subspace $\ker dF$ (precisely, the matrix expression for $h_v$ is in the basis $\{Z_v, Z_1,\ldots, Z_{n-1}\}$ where the $Z_i$ are any independent system spanning $\ker dF$. In particular, note that $h_v(Z_v,Z_v)=dF\cdot Z_v$ and $h_v(Z_v,Z_i)=0$, $1 \leq i \leq n-1$).

The above construction is such that $h_v$ depends continuously on $v$, $h_0=h$ and $h_v=h$ in $M-K$. Finally, we show that $\grad^{h_v} F = Z_v$. To this end, let $W$ be an arbitrary vector field; we can decompose it uniquely as $W=a_1Z_v + W_h$, where $W_h \in \ker dF$, $a_1 \in \R$. We then have \begin{align}dF \cdot W &= dF \cdot (a_1Z_v + W_h)\\ &= a_1 dF \cdot Z_v = a_1 h_v(Z_v,Z_v) 
\\&= h_v(a_1Z_v+W_h,Z_v)=h_v(W,Z_v),  \end{align} which concludes the proof.
\end{proof}

Now introduce the flow map of~\eqref{eq:caucy0} \begin{equation}\label{eq:defPhi}\Phi:\R^n \to M:v \mapsto \Phi(v):=\phi^{v}_{r'}(y_1).\end{equation}	 Then, recalling that $\nabla^hF=(1,0\ldots0)$ in $K$, we see that $\Phi(0)=y_1+(r',0,\ldots,0)=:y_2 \in \partial K.$
Furthermore, we have the following Lemma:
	  
	  \begin{figure}[t]\centering
	  \includegraphics[width=.45\columnwidth]{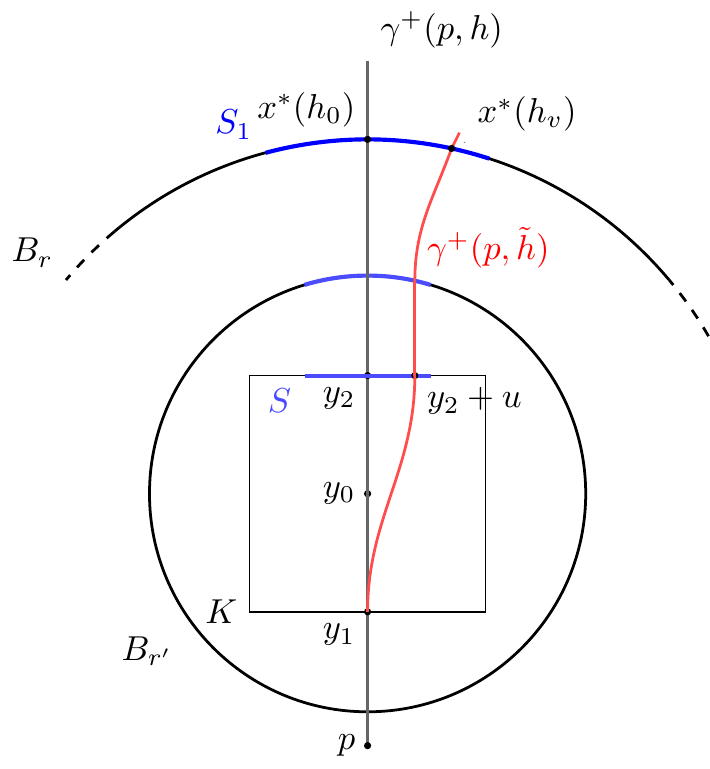}
	  \caption{\small The gradient flow inside $B_{r'}$ goes along vertical lines. The principal flow line for $h$, and for any metric agreeing with $h$ outside of $K$, contains the segment $(p,y_1)$. Changing the metric only inside $K$, we can make the corresponding principal flow line go through $y_2+u$, for any small $u$. The set of realizable intersections of top face of $K$ and principal flow lines (by changing the metric to $h'$ inside $K$) is denoted by $S$. Since $S$ is transversal to $\grad^h F$ (and thus to $\grad^{h'}F$, since $h$ and $h'$ agree outside of $K$), its image under the gradient flow intersects $\partial B_r$ to yield $S'$ containing an open set around $x^*(h_0)$. }
	  \label{fig:flowK}	
	  \end{figure}
\begin{lemma}\label{lem:philocallysurjective}
The map $\Phi$ defined in Eq.~\eqref{eq:defPhi} is locally surjective around $0$.	
\end{lemma}
\begin{proof}
We prove the statement by showing that the linearization of $\Phi$ around $0$ is surjective. Denote by $w(t)=y_1+te_1$ the solution of~\eqref{eq:caucy0} with $v=0$. It is clear that $w(t)$ is a segment of the positive principal flow line of $\grad^h F$ at $p_i$, and that $w(0)=y_1=((r-r')/2,0,\ldots,0)$ and $w(r')=((r+r')/2,0,\ldots,0)=y_2$. Recall the perturbation formula~\cite[Sec. 32]{abraham1967transversal} 
	\begin{equation}\label{eq:variation}
	\frac{d}{d\eta}|_{\eta=0} \Phi(\eta v) = \int_0^{r'} Y_{v}(w(r'-s))ds.
	\end{equation}
In particular, the right-hand side depends on the value of $Y_v$ along $w$ {\it only} and, by construction, is equal to $v$. This proves that $\Phi$ is locally surjective as claimed.
\end{proof}

To conclude the proof, we show for any $\delta>0$, we we can find  $v$ with $\|v\|<\delta$ so that the gradient of $F$ for $h_v$ is simple. Since $h_0=h$ and $h_v$ is continuous in $v$, this shows that there exists metric arbitrarily close to $h$ for which the gradient of $F$ is simple.	

As above, let $\cU_i \ni h$ be  the open neighborhood of $h$ from Lemma~\ref{lem:contprincipal}. By perhaps decreasing $\delta$, we can ensure that $h_v \in \cU_i$ for all $v$ with $\|v\| < \delta$ (since $h_v$ depends continuously on $v$, and $h_0=h$.)
	
	Denote by $x^*(h_v) \in \partial B_r$ the point of intersection of $\partial B_r$ and $\gamma(p_i,h_v)$ (the intersection is not empty per Lemma~\ref{lem:contprincipal}). Then for each $\|v\|<\delta$, $x^*(h_v)$ is on the same flow line as $y_1$ since $h_v$ agrees with $h$ outside of $K$. Because $\Phi$ is locally surjective around $0$, appealing to the inverse function theorem, we can find, for $\delta$ and  $\delta_1$ small enough, a continuous function $\Phi^{-1}: u \mapsto v$ so that $\Phi(v)=u$, for all $u$ with $\|u\|<\delta_1$. Let $S \in \partial K$ be the subset of the 'top face' defined as   $$S:=\{y_2+u \mid \|u\|<\delta_1 \mbox{ and } e_1^\top u =0 \}.$$  Note that $S$ and $\nabla^{h}F$ are transversal by construction.

To make the notation simpler, we set $v:=\Phi^{-1}(u)$.  The principal flow line of $\grad^{h_{v}}$ intersects $S$ at $y_2+u$: every point in $S$ can thus be made to belong to a principal flow line of a $\grad^{h_v}F$ for an appropriate $v$. Using again the fact that $h_{v}$ agrees with $h_0=h$ outside of $K$, we see that 
$$e^{-{[0,\infty)}\grad^{h_{v}} F}(S)=e^{-{[0,\infty)}\grad^{h} F}(S),$$
  and thus  $S_1:=e^{-{[0,\infty)}\grad^{h_{v}} F}(S) \bigcap \partial B_r$ contains an open set around $x^*(h_0)$. 
    Hence, for any $x_1^* \in \partial B_r$ near $x^*(h_0)$, we can find a $v$ so that the principal flow line of $\grad^{h_v}F$ goes through $x_1^*$. Finally, since $W_n(p_i,h) = W_n(p_i,h_v)$ and $W_n(p_i,h) \cap \partial B_{r}$ is closed, there exists  $x_1^* \in \partial B_r$ arbitrarily close to $x^*(h)$---and thus a $v$ arbitrarily small---so that the principal flow line of $\grad^{h_v}F$ does not belong  to $W_n(p_i, h_v) \cap \partial B_r$ and thus does not belong to $W_n(p_i, h_v)$.  This concludes the proof.
    \end{proof}

\section{Summary and outlook: max-min graphs}

\subsection{Max-min graphs} From the main result of the paper, we see that given a smooth $n$-dimensional closed manifold $M$,  to any generic pair $(F,g) \in C^\infty(M) \times \cM$, there is a naturally assigned bipartite graph $G=(V,E)$, which we call {\em max-min graph} of $(F,g)$
\begin{definition}[Max-min graph of $(F,g)$] The max-min graph of a generic pair $(F,g) \in C^\infty(M) \times \cM$ is the bipartite graph $G=(V,E)$ with $V=\crit_0(F) \cup \crit_n(F)$ and  

$$ E = \{(p_i,m_-(p_i)),(p_i,m_+(p_i))\mid p_i \crit_n(F) \}.$$
\end{definition}

The set of possible max-min graphs for generic gradient vector fields for $n=1$ is easily seen to depend on the topology of $M$, and can be  completely characterized: denote by $p_i$ and $m_j$ the elements of $\crit_n(F)$ and $\crit_0(F)$, respectively. Denote by $H_1(M)$ is the first homology group of $M$ which, since $\dim M=1$ and $M$ is connected, has rank either $0$ or $1$. Recall that if $M=\R^n$, we assume that $\lim_{\|x\| \to \infty} F(x)=\infty$ and $F$ has a finite number of critical points. We have (see Fig.~\ref{fig:figoneD} for an illustration)

\begin{figure}[t]
\centering
\includegraphics[width=.8\columnwidth]{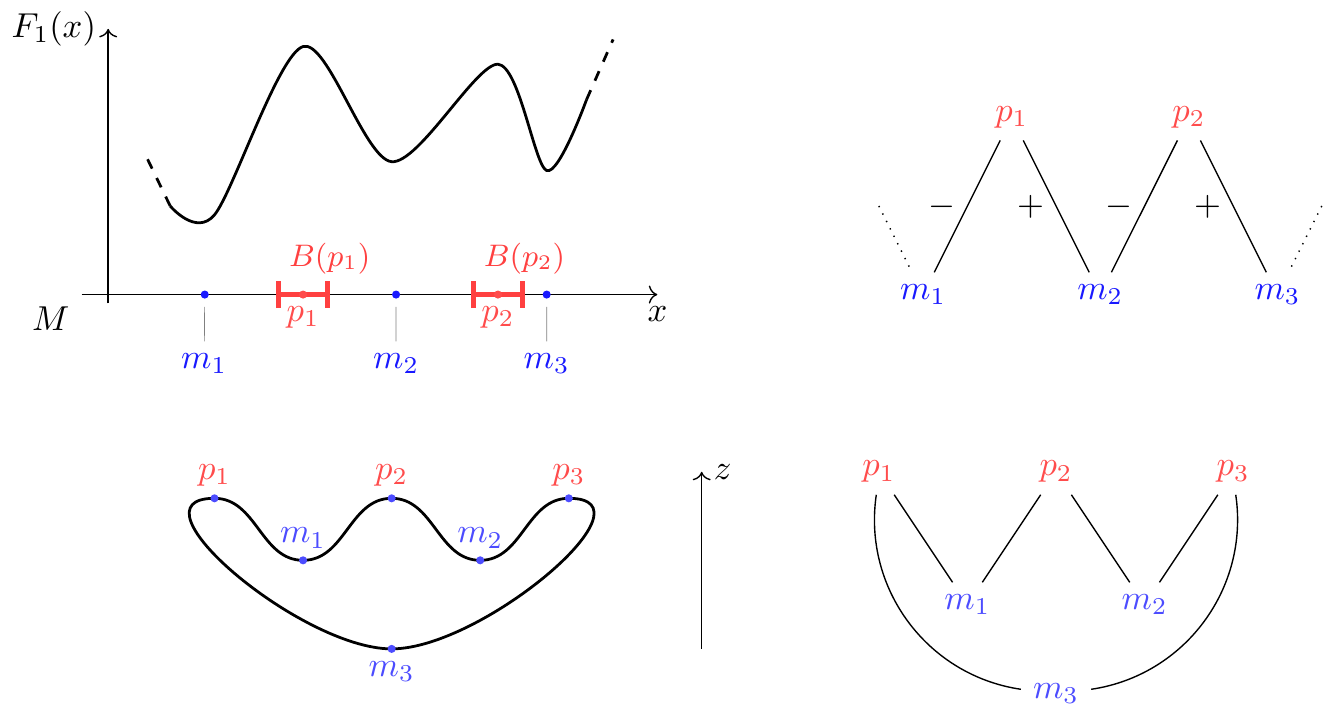}
\caption{\small {\it Top Left:} A Morse function on $M=\R$. To each maximum $p_i$, we can assign two minima $m_-(p), m_+(p)$ so that  gradient descent for $F_1$ initialized in $B(p)$ converges to either $m_-(p)$ or $m_+(p)$. {\it Top Right:} The flow graph of the gradient of $F_1$. Each minimum has degree two and maxima have degrees one or two. {\em Bottom Left:} A circle is embedded in the plane with vertical axis $z$ and we consider the Morse function $F_2(x)=z$ (height function). It has three maxima and three minima. {\em Bottom Right:} The flow graph of the gradient of $F_2$. All critical points have degree two.}	
\label{fig:figoneD}
\end{figure}

\begin{proposition}[Max-min graphs for $\dim M= 1$]\label{prop:casedim1}
Assume $\dim M =1$, then
\begin{enumerate}
\item {\em case $\operatorname{rank} H_1(M)=0$:}  $k=:|\crit_0(F)|=|\crit_n(F)|+1$ and there exists an ordering of $p_i$, $m_j$ so that $$E=\cup_{i=1}^{k}\{(p_i,m_i),(p_i,m_{i+1})\}.$$ Thus $\deg(p_i)=2$ for $p_i \in \crit_n(F)$.
\item {\em case $\operatorname{rank} H_1(M)=1$:} then $k=:|\crit_n(F)|=|\crit_0(F)|$ and there exists an ordering of $p_i$, $m_j$ so that $$E=\cup_{i=1}^{k}\{(m_i,p_i),(m_i,p_{i+1\!\! \mod k})\}.$$ Thus $\deg(m_i)=\deg(p_i)=2$ for $m_i,p_i \in V$.
\end{enumerate}
\end{proposition}
The proof of the proposition is an immediate consequence of the following facts: (1)  $F$ is generically Morse (and thus does not have saddle points if $\dim M =1$); (2)  the critical points of $F$ can in this case be given a cyclic (if $\operatorname{rank} H_1(M)=1$) or linear (if $\operatorname{rank} H_1(M)=0$) order and (3) maxima and minima of $F$ appear alternatively in this order.

\subsection{Realizable max-min graphs and topology of $M$} This leads us to the following:

\vspace{.1cm}
\noindent{\bf Open problem}: what kind of bipartite graphs can be max-min graphs of a pair $(F,g)$ over $M$?
\vspace{.1cm}

To address this problem, we call an {\em abstract max-min} graph any simple bipartite graph $G=(V_0 \cup V_1, E)$ where 
\begin{enumerate}
\item $|V_0|\geq 1$, $|V_1|\geq 1$
\item 	$1 \leq \deg(p) \leq 2$ for all $p \in V_1$
\end{enumerate}
We think of $V_0$ as the set of minima and $V_1$ as the set of maxima. We say that a pair $(F,g)$  {\em realizes} $G$ on $M$ with the max-min graph of $\nabla^g F$ is equal to $G$.

The set of abstract max-min graphs that can be realized depends on the topology of $M$, as was clear in the case $\dim M=1$ described in Prop.~\ref{prop:casedim1}. We can also easily realize max-min graphs with a single node in $V_0$ and an arbitrary number of nodes in $V_1$, by generalizing the construction of Fig.~\ref{fig:figexD} to add more maxima. These yield max-min graphs where the degree of elements in $V_1$ is one and the degree of the element in $V_0$ is unbounded. Reciprocally, we can have functions with a single node in $V_1$ and an arbitrary number of nodes in $V_0$. For example, it suffices to consider the negative of the height function for the embedded sphere in Fig.~\ref{fig:figexD}. From this particular example, we also conclude that flow graphs can be disconnected: since $|V_0|=3$ and $|V_1|=1$ and the degree of the node in $V_1$ is at most 2, at least one node in $V_0$ has no incident edges. Furthermore, we see that  reversing the direction of the gradient flow (i.e., considering the gradient ascent flow of $F$ instead of the gradient descent flow), does {\em not} yield an automorphism of the corresponding flow graphs: indeed, while the elements of $V_0$ become the elements of $V_1$ and vice-versa, the edge sets of the two flow graphs do not even necessarily have the same cardinality. Finally, it should be clear that none of the examples described in the paragraph could be realized over a state-space $M$ of dimension $1$. The above leads to the question of how can one realize an abstract max-min graph, and what restriction on the topology of the underlying state-space is imposed. We will address these questions, and others, in a forthcoming publication.

\begin{figure}[t]\centering
\includegraphics[width=.5\columnwidth]{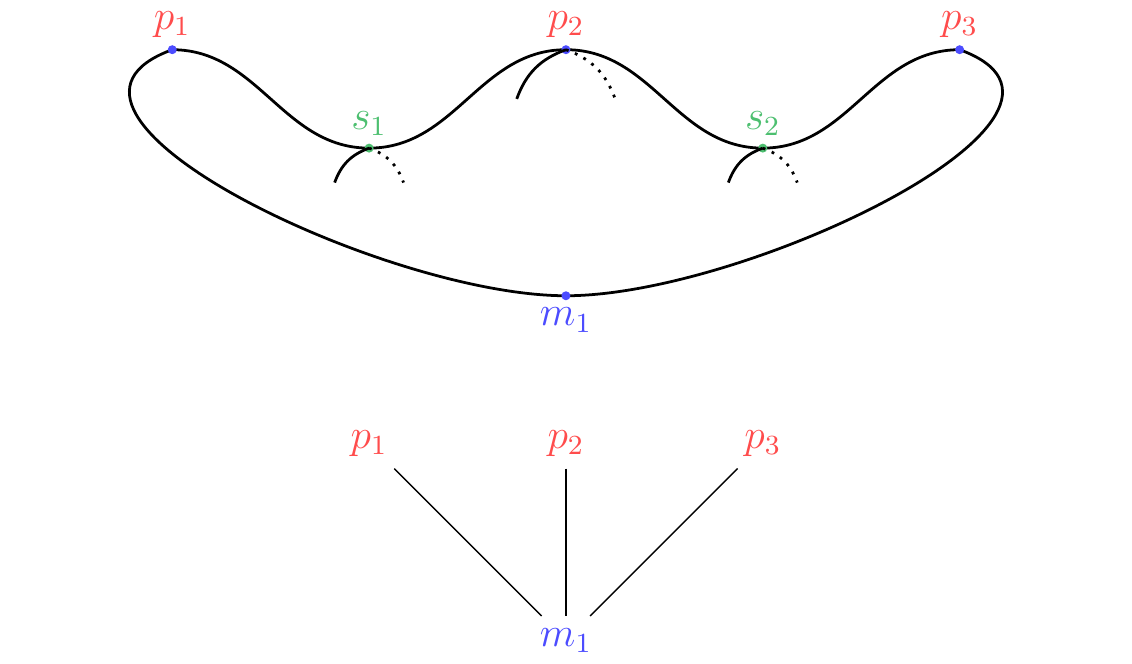}
\caption{\small {\it Top:} We consider the height functions of an embedded sphere in $\R^3$. The function  has three maxima $p_1,p_2,p_3$, two saddle points $s_1,s_2$ and a minimum $m_1$. {\it Bottom:} Max-min graph of the gradient of the height function of the embedded sphere.}	 
\label{fig:figexD}
\end{figure}
\subsection{Summary}

Let $M$ be a smooth closed manifold and $(F,g)$ a generic pair where $F$ is a smooth function and $g$ a Riemannian metric on $M$. We have shown in this paper that to each maximum $p$ of $F$, we can assign two minima---denoted $m_-(p),m_+(p)$---having the following property: the gradient flow of $F$ initialized close enough to $p$ converges with high-probability to the set $\{m_-(p), m_+(p)\}$. In order to prove the result, we introduced the notion of {\it principal flow lines} of a maximum. When  the linearization of the gradient flow around  $p$ has a smallest eigenvalue of algebraic multiplicity one, we showed the existence of exactly two flow lines of the gradient ascent flow that reach $p$ tangentially to the corresponding eigenspace. These are the principal flow lines of $p$. If they belong to the stable manifolds of minima of $F$, we call the corresponding gradient vector field {\it simple}. We then showed in a first part that for simple gradients, most of the volume of any small ball containing at maximum $p$ belongs to the union of the two stable manifolds to which principal flow lines belong.  In a second part, we showed that simple gradient vector fields are generic. 

The proof of the first part is local in nature, with the exception of the reliance on the global stable manifold decomposition theorem. The $C^1$ linearization result of Hartman~\cite{hartman1960mex} plays an important role there, and we note that it holds only if all eigenvalues of the linearized gradient vector field  have real parts of the same sign. This result thus cannot be used at a saddle point of $F$. We also point out that the topological equivalence provided by the Hartman-Grobman theorem, which can be applied at any hyperbolic fixed point, is not sufficient to obtain our result.   The second part of the proof shows that generically for $(F,g)$, the linearization of the gradient flow at a maximum has a smallest eigenvalue of multiplicity one, and  the  corresponding principal flow lines   belong to the stable manifolds of some minima. The proof that the linearization of the gradient vector field at $p$ has a unique smallest eigenvalue relies on transversality arguments. The proof that the principal flow lines belong to stable manifolds of minima goes by showing that the property holds for an increasing sequence of compact subsets of the stable manifolds, and appealing to Baire theorem.
Finally, we introduced the notion of max-min graph graph of a generic pair $(F,g)$, and described some of its properties along with open questions.
\section{Appendix}

\begin{lemma}\label{lem:techmatrix1}
Let $A,B \in \R^{n \times n}$ be positive definite matrices so that $AB$ has repeated eigenvalues. Then for any $\varepsilon>0$, there exists a positive definite $Q$, with $\|Q\|<\varepsilon$ and $(A+Q)B$ has distinct eigenvalues.
\end{lemma}
\begin{proof}
We give a simple, constructive proof. The matrix $AB$ is similar to $B^{1/2}AB^{1/2}$. The latter being symmetric, there exists an orthogonal matrix $P$ and a diagonal matrix $D$ so that $P^\top B^{1/2}AB^{1/2}P = D$, where the diagonal entries of $D$ are the eigenvalues of $AB$. Denote the $p_i \in \R^n$ the $i$th column of $P$. Then $p_i^\top p_j = \delta_{ij}$ and $P^\top p_i = e_i$. Now  set $v_i = B^{-1/2} p_i$. Then  $P^\top B^{1/2}(A+\varepsilon_i v_iv_i^\top)B^{1/2}P= D+\varepsilon_i e_ie_i\top.$ Since $D+\varepsilon_i e_ie_i^\top$ is diagonal, it contains the eigenvalues of  $P^\top B^{1/2}(A+\varepsilon_i v_iv_i^\top)B^{1/2}P$, which are the same as the eigenvalues of $(A+\varepsilon_i v_iv_i^\top)B$. It now suffices to choose the $\varepsilon_i>0$ small enough  and so that $D+\diag(\varepsilon_1,\ldots,\varepsilon_n)$ has distinct entries, and set $Q=\sum_{i=1}^n \varepsilon_i v_iv_i^\top.$	
\end{proof}

\begin{lemma}\label{lem:conteig} Let $A \in \R^{n \times n}$ be a real symmetric  matrix with eigenvalues $\lambda_1 > \lambda_2 \geq \cdots \geq \lambda_n$. Let $v_1:S_n \to \R  P^{n-1}:A \mapsto v_1(A)$ be a map assigning to $A$ the eigenspace associated with $\lambda_1$. Then  $v_1$ is differentiable around $A$.
\end{lemma}
\begin{proof}
Consider the map $$V:S_n  \times \R^{n} \times \R \to \R^{n+1}	:(X,u,\lambda) \mapsto \begin{pmatrix}(\lambda I-X)u\\
 u^\top u -1	
 \end{pmatrix}.$$
Let $A \in S_n$ be such that $\lambda_1 > \lambda_2$ and denote by $v$ a unit eigenvector spanning the eigenspace of $\lambda_1$. Then $V(A,v,\lambda_1)=0$ and  the differential of $V$ with respect to $u,\lambda$ evaluated at $(A,v,\lambda_1)$ is $$d_{u,\lambda}V(A,v,\lambda_1) = \begin{pmatrix}
 \lambda_1I-A & v \\
 2v^\top & 0 
 \end{pmatrix}.$$ Since $\lambda_1$ is a simple eigenvalue of $A$, the above map is invertible.  Hence, the implicit function theorem states that there is an open set $U \subset S_n$ containing $A$ and differentiable functions $\lambda(X),u(X)$ such that $(\lambda(X)I-X)u(X)=0$ and $\|u(X)\|^2=1$ for all $X \in U$, which proves the result.
\end{proof}

\begin{lemma}\label{lem:measure}
Let $B_\delta = B^1_\delta \cup B^2_\delta$ and $V=V^1 \cup V^2$ with $$\mu(B^1_\delta \cap B^2_\delta)=\mu(V^1\cap V^2)=\mu(B^1_\delta \cap V_2)=\mu(B^2_\delta \cap V_1)=0.$$ Assume that $$\lim_{\delta \to 0} \frac{\mu(B^1_\delta \cap V^1)}{\mu(B^1_\delta)} =\lim_{\delta \to 0} \frac{\mu(B^2_\delta \cap V^2)}{\mu(B^2_\delta)}=1.$$ Then it holds that $$\lim_{\delta \to 0} \frac{\mu(B_\delta \cap V)}{\mu(B_\delta)}=1$$
\end{lemma}

\begin{proof}
	Since $B^i_\delta \subseteq B_\delta$, we have
	\begin{align*}
	0&=1-\lim_{\delta \to 0} \frac{\mu(B^1_\delta \cap V^1)}{\mu(B^1_\delta)}\\
	&=\lim_{\delta \to 0} \frac{\mu(B^1_\delta-(B^1_\delta \cap V^1))}{\mu(B^1_\delta)}=\lim_{\delta \to 0} \frac{\mu(B^1_\delta-(B^1_\delta \cap V^1))}{\mu(B_\delta)}	
	\end{align*}
and, similarly, $\lim_{\delta \to 0} \frac{\mu(B^2_\delta-(B^2_\delta \cap V^2))}{\mu(B_\delta)}=0$. Summing the above two equalities, we get in the numerator

\begin{align*}
\mu(B^1_\delta-(B^1_\delta \cap V^1))+\mu(B^2_\delta-(B^2_\delta \cap V^2))&=
\mu\left(	(B^1_\delta-(B^1_\delta \cap V^1))\cup(B^2_\delta-(B^2_\delta \cap V^2))\right)\\
&=\mu\left(	(B^1_\delta-(B^1_\delta \cap V))\cup(B^2_\delta-(B^2_\delta \cap V))\right)\\
&=\mu\left(	B_\delta-((B^1_\delta \cap V))\cup(B^2_\delta \cap V))\right)\\
&=\mu\left(B_\delta-(B_\delta \cap V)\right).
\end{align*}
Hence, $\lim_{\delta\to 0} \frac{\mu\left(B_\delta-(B_\delta \cap V)\right)}{\mu(B_\delta)}=0,$ which concludes the proof.
\end{proof}

\bibliographystyle{amsplain}
\bibliography{morse2min.bib}
\end{document}